\documentclass{article}
\usepackage{amsmath,amscd,amsfonts,amssymb,amsthm}
\usepackage{graphicx}
\usepackage{epstopdf}
\usepackage{a4wide}
\usepackage{graphicx}
\usepackage{subfigure}
\usepackage[latin1]{inputenc}
\usepackage{color}

\newcommand{\subj}[1]{\par\noindent{\bf Mathematics Subject Classification 2010: }#1.}
\newcommand{\keyw}[1]{\par\noindent{\bf Keywords: }#1.}

\theoremstyle{definition}
\newtheorem{definition}{Definition}
\newtheorem{theorem}{Theorem}

\newtheorem{lemma}{Lemma}
\theoremstyle{remark}

\def\a{\alpha}
\def\r{\rho}
\def\t{\tau}
\def\DS{\displaystyle}
\def\LD{{^C\mathcal{D}_{a+}^{\a,\r}}}
\def\LDRL{{\mathcal{D}_{a+}^{\a,\r}}}
\def\LI{{\mathcal{I}_{a+}^{\a,\r}}}
\def\RD{{^C\mathcal{D}_{b-}^{\a,\r}}}
\def\RDRL{{\mathcal{D}_{b-}^{\a,\r}}}
\def\RI{{\mathcal{I}_{b-}^{\a,\r}}}

\begin{document}

\title{A Gronwall inequality for a general Caputo fractional operator}

\author{Ricardo Almeida\\
{\tt ricardo.almeida@ua.pt}}

\date{$^1$Center for Research and Development in Mathematics and Applications (CIDMA)\\
Department of Mathematics, University of Aveiro, 3810--193 Aveiro, Portugal}

\maketitle


\begin{abstract}

In this paper we present a new type of fractional operator, which is a generalization of the Caputo and Caputo--Hadamard fractional derivative operators. We study some properties of the operator, namely we prove that it is the inverse operation of a generalized fractional integral. A relation between this operator and a Riemann--Liouville type is established. We end with a fractional Gronwall inequality type, which is useful to compare solutions of fractional differential equations.

\end{abstract}

\subj{26A33, 34A08, 34A40}

\keyw{fractional calculus, fractional differential equations, Gronwall inequality}


\section{Introduction}

Fractional calculus is an important subject with numerous applications to different fields outside mathematics like physics \cite{Carpintery,Mainardi,West}, chemistry \cite{Bagley,Douglas,Kaplan}, biology \cite{Arafa,Magin,Sebaa,Xu}, engineering \cite{Duarte,Feliu,Ortiguera,Silva}, etc. It allows us to define derivatives and integrals of non-integer order, which may be more suitable to model real world phenomena, and nowadays this subject is not only important in mathematics, but also by its numerous employments in applicable sciences. We find in the literature several definitions for fractional operators, although the most important ones are the Riemann--Liouville, the Caputo and the Grunwald--Letnikov fractional derivatives \cite{Kilbas,Samko}. The choice of the best operator depends on the analysis of the system, and because of this we find a vast work dealing with different operators, for similar subjects. To overcome this situation, one solution is to consider general definitions for fractional operators, for which we can recover the classical ones as particular cases. For example, using general Kernels we can obtain some of the most important fractional operators \cite{Odzijewicz1,Odzijewicz2}. In Section \ref{sec:FC} we do a review of some of the most important notions when dealing with fractional derivatives and integrals.
In \cite{Katugampola1,Katugampola2,Katugampola3}, U. Katugampola presents a general form of fractional operator, by introducing a new parameter $\r>0$, for which we can obtain the Riemann--Liouville fractional operators when $\r=1$, and the Hadamard fractional operators as $\r\to0^+$. Later, in \cite{Almeida}, the authors present a Caputo type fractional derivative of order $\a\in(0,1)$, and some properties are proven. In this work, we start by defining a Caputo--Katugampola fractional derivative of arbitrary real order $\a>0$, and as we shall see it is the inverse operator of the Katugampola fractional integral. Several properties of the new fractional derivative operator are studied in Section  \ref{sec:CK}. To end, in Section \ref{sec:Gronwall}, we present and prove a fractional Gronwall inequality type, generalizing the ones presented in  \cite{Lin,Qian,Ye}.

\section{Preliminaries on fractional calculus}\label{sec:FC}

Let $x:[a,b]\to\mathbb R$ be an integrable function.
Starting with the Cauchy's formula for an n-fold integral
$$\int_a^t d\t_1\,\int_a^{\t_1}d\t_2\,\ldots \int_a^{\t_{n-1}}x(\t_n) \, d\t_n=\frac{1}{(n-1)!}\int_a^t (t-\t)^{n-1}x(\t) d\t,$$
we find a direct generalization for integrals of arbitrary real order $\a>0$. The Riemann-Liouville fractional integral of order $\a$ of $x$ is defined as
$${I_{a+}^\a}x(t)=\frac{1}{\Gamma(\a)}\int_a^t (t-\t)^{\a-1}x(\t) d\t,$$
where $\Gamma(\cdot)$ denotes the Gamma function.
Later, by considering the formula
$$\int_a^t \frac{1}{\t_1}\, d\t_1\,\int_a^{\t_1} \frac{1}{\t_2}\, d\t_2\,\ldots \int_a^{\t_{n-1}} \frac{1}{\t_n}x(\t_n)\, d\t_n
=\frac{1}{(n-1)!}\int_a^t \left(\ln\frac{t}{\t}\right)^{n-1}\frac{x(\t)}{\t} d\t,$$
Hadamard defined a new type of fractional operator, known nowadays as Hadamard fractional integral \cite{Hadamdard}:
$${^HI_{a+}^\a}x(t)=\frac{1}{\Gamma(\a)}\int_a^t \left(\ln\frac{t}{\t}\right)^{\a-1}x(\t) \frac{d\t}{\t}.$$
Fractional derivatives are defined using the fractional integral operators. Beginning with the Riemann--Liouville fractional integral, we find the most important definitions for fractional derivatives. Let $\a>0$ and $n\in\mathbb N$ be such that $\a\in(n-1,n)$. The Riemann--Liouville fractional derivative of order $\a$ of a function $x$ is defined as
$${D_{a+}^\a}x(t)=\left(\frac{d}{dt}\right)^n{I_{a+}^{n-\a}}x(t)=\frac{1}{\Gamma(n-\a)}\left(\frac{d}{dt}\right)^n\int_a^t (t-\t)^{n-\a-1}x(\t) d\t,$$
while the Caputo fractional derivative is defined as
$${^CD_{a+}^\a}x(t)={I_{a+}^{n-\a}}\left(\frac{d}{dt}\right)^n x(t)=\frac{1}{\Gamma(n-\a)}\int_a^t (t-\t)^{n-\a-1}\left(\frac{d}{d\t}\right)^nx(\t) d\t.$$
For what concerns the Hadamard fractional derivative, we have
$${^H D_{a+}^\a}x(t)=\left(t\frac{d}{dt}\right)^n{^HI_{a+}^{n-\a}}x(t)=\frac{1}{\Gamma(n-\a)}\left(t\frac{d}{dt}\right)^n\int_a^t
\left(\ln\frac{t}{\t}\right)^{n-\a-1}x(\t)  \frac{d\t}{\t}.$$
For more on the subject, we advice the reader to \cite{Kilbas,Samko}.
Finally, in \cite{Baleanu1,Baleanu2}, the Caputo--Hadamard fractional derivative is presented and some properties studied, and the definition is
$${^{CH}D_{a+}^\a}x(t)={^HI_{a+}^{n-\a}} \left(t\frac{d}{dt}\right)^nx(t)=\frac{1}{\Gamma(n-\a)}\int_a^t
\left(\ln\frac{t}{\t}\right)^{n-\a-1}\left(\t\frac{d}{d\t}\right)^nx(\t) \frac{d\t}{\t}.$$
The previous notions can be generalized by introducing a new parameter in the definitions, and for some particular cases we recover the classical ones.
In \cite{Katugampola1}, starting with the formula
$$\int_a^t \t_1^{\r-1}\, d\t_1\,\int_a^{\t_1} \t_2^{\r-1}\, d\t_2\,\ldots \int_a^{\t_{n-1}} \t_n^{\r-1}x(\t_n)\, d\t_n
=\frac{\r^{1-n}}{(n-1)!}\int_a^t \t^{\r-1}(t^\r-\t^\r)^{n-1}x(\t) d\t,$$
Katugampola suggest a new type of fractional integral, which includes the Riemann--Liouville type by considering $\r=1$ and the Hadamard integral when $\r\to0^+$.

\begin{definition} Let $a,b>0$ be two reals, and $x:[a,b]\rightarrow\mathbb{R}$ be an integrable function. The left-sided and right-sided Katugampola fractional integrals of order $\a>0$ and parameter $\r>0$ are defined respectively by
$$\LI x(t)=\frac{\r^{1-\a}}{\Gamma(\a)}\int_a^t \t^{\r-1}(t^\r-\t^\r)^{\a-1}x(\t) d\t$$
and
$$\RI x(t)=\frac{\r^{1-\a}}{\Gamma(\a)}\int_t^b \t^{\r-1}(\t^\r-t^\r)^{\a-1}x(\t) d\t.$$
\end{definition}

Also, in \cite{Katugampola2}, a differential operator of order $\a>0$ with dependence on a parameter $\r>0$ is defined as
$$\LDRL x(t)=\left(t^{1-\r}\frac{d}{dt}\right)^n{\mathcal{I}_{a+}^{n-\a,\r}}x(t)=
\frac{\r^{1-n+\a}}{\Gamma(n-\alpha)}\left(t^{1-\r}\frac{d}{dt}\right)^n\int_a^t \t^{\r-1}(t^\r-\t^\r)^{n-\a-1}x(\t) d\tau,$$
for the left-sided fractional derivative, and for the right-sided fractional derivative we have
$$\RDRL x(t)=\left(-t^{1-\r}\frac{d}{dt}\right)^n{\mathcal{I}_{b-}^{n-\a,\r}}x(t)=
\frac{\r^{1-n+\a}}{\Gamma(n-\alpha)}\left(-t^{1-\r}\frac{d}{dt}\right)^n\int_t^b \t^{\r-1}(\t^\r-t^\r)^{n-\a-1}x(\t) d\tau.$$

\section{Caputo--Katugampola fractional derivative}\label{sec:CK}

Having in mind the different definitions for fractional operators, a notion of Caputo--Katugampola fractional derivative is immediate.

\begin{definition}\label{CKFD}
Let $0< a<b<\infty$ be two reals, $\r$ be a positive real number, $\a\in\mathbb R^+$ and $n\in\mathbb N$ be such that $\a\in(n-1,n)$, and $x:[a,b]\rightarrow\mathbb{R}$ a function of class $C^n$. The left-sided and right-sided Caputo--Katugampola fractional derivatives of order $\a$ and parameter $\r$ are defined respectively by
$$\LD x(t)= {\mathcal{I}_{a+}^{n-\a,\r}}\left(t^{1-\r}\frac{d}{dt}\right)^nx(t)
=\frac{\r^{1-n+\a}}{\Gamma(n-\alpha)}\int_a^t \t^{\r-1}(t^\r-\t^\r)^{n-\a-1}\left(\t^{1-\r}\frac{d}{d\t}\right)^nx(\t) d\tau$$
and
$$\RD x(t)= {\mathcal{I}_{b-}^{n-\a,\r}}\left(-t^{1-\r}\frac{d}{dt}\right)^nx(t)
=\frac{\r^{1-n+\a}}{\Gamma(n-\alpha)}\int_t^b \t^{\r-1}(\t^\r-t^\r)^{n-\a-1}\left(-\t^{1-\r}\frac{d}{d\t}\right)^nx(\t) d\tau.$$
\end{definition}

We refer to \cite{Almeida} for a detailed study when $\a\in(0,1)$. Also, since the left-sided and right-sided Katugampola fractional integrals are bounded linear operators, it is clear that the left-sided and right-sided Caputo--Katugampola fractional derivatives are continuous operators on the closed interval $[a,b]$.
From the definition, it is obvious that the fractional derivative of a constant is zero.

In order to simplify the writing, we introduce the notation
$$x_{(n)}(t):=\left(t^{1-\r}\frac{d}{dt}\right)^nx(t).$$
Let $C^n[a,b]$ be the set of functions $x$ such that $x^{(n)}$ exists and is continuous on $[a,b]$. We define on $C^n[a,b]$ the norms
$$\|x\|^\r_{C^n}=\sum_{k=0}^n\max_{t\in[a,b]}|x^{(n)}(t)| \quad \mbox{and} \quad \|x\|_{C}=\max_{t\in[a,b]}|x(t)|.$$

\begin{theorem} The following relations hold:
$$\lim_{\a\to n^-}\LD x(t)=x_{(n)}(t), \quad \lim_{\a\to (n-1)^+}\LD x(t)=x_{(n-1)}(t)-x_{(n-1)}(a),$$
$$\lim_{\a\to n^-}\RD x(t)=(-1)^nx_{(n)}(t), \quad \lim_{\a\to (n-1)^+}\RD x(t)=(-1)^n(x_{(n-1)}(b)-x_{(n-1)}(t)),$$
\end{theorem}
\begin{proof} Integrating by parts, we deduce
$$\begin{array}{ll}
\DS\LD x(t)&=\DS\frac{\r^{1-n+\a}}{\Gamma(n-\alpha)}\int_a^t \t^{\r-1}(t^\r-\t^\r)^{n-\a-1}x_{(n)}(\t) d\tau\\
&=\DS\frac{\r^{-n+\a}}{\Gamma(n+1-\alpha)}(t^\r-a^\r)^{n-\a}x_{(n)}(a)+\frac{\r^{-n+\a}}{\Gamma(n+1-\alpha)}\int_a^t (t^\r-\t^\r)^{n-\a}
\frac{d}{d\t}x_{(n)}(\t) d\tau.\\
\end{array}$$
Thus,
$$\lim_{\a\to n^-}\LD x(t)=x_{(n)}(a)+[x_{(n)}(t)-x_{(n)}(a)]=x_{(n)}(t).$$
For the second formula, starting with the definition, we obtain directly that
$$\begin{array}{ll}
\DS \lim_{\a\to (n-1)^+}\LD x(t)&=\DS \int_a^t \frac{d}{d\t}x_{(n-1)}(\t) d\tau\\
&=\DS  x_{(n-1)}(t)-x_{(n-1)}(a).
\end{array}$$
The other two formulas are proven in a similar way.
\end{proof}

The following result is easily proven, and we omit the proof here.

\begin{theorem}\label{teo:bound} Given a function $x\in C^n[a,b]$ and $t\in[a,b]$, we have
$$|\LD x(t)|\leq \frac{\r^{\a-n}}{\Gamma(n+1-\a)}\max_{\t\in[a,t]}|x_{(n)}(\t)|(t^\r-a^\r)^{n-\a}$$
and
$$|\RD x(t)|\leq \frac{\r^{\a-n}}{\Gamma(n+1-\a)}\max_{\t\in[t,b]}|x_{(n)}(\t)|(b^\r-t^\r)^{n-\a}.$$
In particular, $\LD x(a)=0$ and $\RD x(b)=0$.
\end{theorem}

\begin{theorem} The fractional derivatives $\LD$ and $\RD$ are bounded operators from $C^n[a,b]$ to $C[a,b]$, with
$$\|\LD x\|_C\leq K \|x\|^\r_{C^n} \quad \mbox{and} \quad \|\RD x\|_C\leq K \|x\|^\r_{C^n},$$
where
$$K= \frac{\r^{\a-n}}{\Gamma(n+1-\a)}(b^\r-a^\r)^{n-\a}.$$
\end{theorem}
\begin{proof} Given $t\in[a,b]$ and  $x\in C^n[a,b]$, using the fact that $|x_{(n)}(t)|\leq \|x\|^\r_{C^n}$ and Theorem \ref{teo:bound}, the result follows.
\end{proof}

\begin{lemma} Consider the functions $x,y:[a,b]\to\mathbb R$ given by
$$x(t)=(t^\r-a^\r)^v, \quad y(t)=(b^\r-t^\r)^v, \quad \mbox{with} \, v>n-1.$$
Then
$$\LD x(t)=\frac{\r^{\a}\Gamma(v+1)}{\Gamma(v-\a+1)}(t^\r-a^\r)^{v-\a}\quad \mbox{and} \quad \RD y(t)=\frac{\r^{\a}\Gamma(v+1)}{\Gamma(v-\a+1)}(b^\r-t^\r)^{v-\a}.$$
\end{lemma}
\begin{proof} We prove only the first one. It is easy to conclude that
$$x_{(n)}(t)=\frac{\r^n\Gamma(v+1)}{\Gamma(v-n+1)}(t^\r-a^\r)^{v-n}.$$
Then,
$$\LD x(t)=\frac{\r^{1+\a}\Gamma(v+1)}{\Gamma(n-\a)\Gamma(v-n+1)}(t^\r-a^\r)^{n-1-\a}
\int_a^t\t^{\r-1}\left(1-\frac{\t^\r-a^\r}{t^\r-a^\r}\right)^{n-1-\a}(\t^\r-a^\r)^{v-n}d\t.$$
With the change of variables
$u=(\t^\r-a^\r)/(t^\r-a^\r)$ and with the help of the Beta function
$$B(x,y)=\int_0^1u^{x-1}(1-u)^{y-1}du, \quad x,y>0,$$
we obtain
$$\LD x(t)=\frac{\r^{\a}\Gamma(v+1)}{\Gamma(n-\a)\Gamma(v-n+1)}(t^\r-a^\r)^{v-\a}B(n-\a,v-n+1).$$
Using the useful property
$$B(x,y)=\frac{\Gamma(x)\Gamma(y)}{\Gamma(x+y)},$$
we prove the formula
$$\LD (t^\r-a^\r)^v=\frac{\r^{\a}\Gamma(v+1)}{\Gamma(v-\a+1)}(t^\r-a^\r)^{v-\a}.$$
\end{proof}

Using these relations, we deduce the fractional derivative of the Mittag--Leffler function
$$E_\a(t)=\sum_{k=0}^\infty\frac{t^k}{\Gamma(\a k+1)}, \, t\in\mathbb R.$$
For all $\lambda \in\mathbb R$, we have
$$\begin{array}{ll}
\DS\LD E_\a(\lambda(t^\r-a^\r)^\a)&\DS=\sum_{k=0}^\infty\frac{\lambda^k}{\Gamma(\a k+1)}\LD (t^\r-a^\r)^{\a k}
=\sum_{k=1}^\infty\frac{\lambda^k}{\Gamma(\a k+1)}\LD (t^\r-a^\r)^{\a k}\\
&\DS=\sum_{k=1}^\infty\frac{\lambda^k}{\Gamma(\a k+1)}\frac{\r^\a\Gamma(\a k+1)}{\Gamma(\a k+1-\a)} (t^\r-a^\r)^{\a k-\a}=\lambda \r^\a E_\a(\lambda(t^\r-a^\r)^\a)\end{array}$$
and
$$\RD E_\a(\lambda(b^\r-t^\r)^\a)=\lambda \r^\a E_\a(\lambda(b^\r-t^\r)^\a).$$

The next two results justify our Definition \ref{CKFD}, since the Caputo--Katugampola fractional derivative is an inverse operation of the Katugampola fractional integral.

\begin{theorem}\label{thm:DerInt} Given a function $x\in C^n[a,b]$, we have
\begin{equation}\label{DerInt}\LI\LD x(t)=x(t)-\sum_{k=0}^{n-1}\frac{\r^{-k}}{k!}(t^\r-a^\r)^kx_{(k)}(a)\end{equation}
and
$$\RI\RD x(t)=x(t)-\sum_{k=0}^{n-1}\frac{\r^{-k}(-1)^k}{k!}(b^\r-t^\r)^kx_{(k)}(b).$$
\end{theorem}

\begin{proof} Using Theorem 4.1 in \cite{Katugampola1}, we have
$$
\begin{array}{ll}
\DS\LI\LD x(t)&=\LI {\mathcal{I}_{a+}^{n-\a,\r}}x_{(n)}(t)={\mathcal{I}_{a+}^{n,\r}}x_{(n)}(t)\\
&\\
&\DS=\frac{\r^{1-n}}{(n-1)!}\int_a^t(t^\r-\t^\r)^{n-1}\frac{d}{d\t}x_{(n-1)}(\t)d\t.
\end{array}$$
Using integration by parts, we deduce
$$\LI\LD x(t)=\frac{\r^{2-n}}{(n-2)!}\int_a^t(t^\r-\t^\r)^{n-2}\frac{d}{d\t}x_{(n-2)}(\t)d\t-\frac{\r^{1-n}}{(n-1)!}(t^\r-a^\r)^{n-1}x_{(n-1)}(a).$$
Integrating again by parts, we have
$$\LI\LD x(t)=\frac{\r^{3-n}}{(n-3)!}\int_a^t(t^\r-\t^\r)^{n-3}\frac{d}{d\t}x_{(n-3)}(\t)d\t-\sum_{k=n-2}^{n-1}\frac{\r^{-k}}{k!}(t^\r-a^\r)^{k}x_{(k)}(a).$$
Repeating this procedure $n-3$ times, we arrive at
$$ \begin{array}{ll}
\DS\LI\LD x(t)&=\DS\int_a^t\frac{d}{d\t}x(\t)d\t-\sum_{k=1}^{n-1}\frac{\r^{-k}}{k!}(t^\r-a^\r)^{k}x_{(k)}(a)\\
&\\
&\DS=x(t)-\sum_{k=0}^{n-1}\frac{\r^{-k}}{k!}(t^\r-a^\r)^kx_{(k)}(a).
\end{array}$$
The second formula is proven is a similar way.
\end{proof}

Taking $\r=1$, formula \eqref{DerInt} reduces to the Caputo case (see e.g. Lemma 2.22 \cite{Kilbas}):
$${I_{a+}^\a}{^CD_{a+}^\a} x(t)=x(t)-\sum_{k=0}^{n-1}\frac{1}{k!}(t-a)^kx^{(k)}(a),$$
and as $\r\to0^+$, having in mind that $\lim_{\r\to0^+}(t^\r-a^\r)/\r=\ln(t/a)$,
we obtain Lemma 2.5 in \cite{Baleanu2}:
$${^H I_{a+}^\a}{^{CH}D_{a+}^\a} x(t)=x(t)-\sum_{k=0}^{n-1}\frac{1}{k!}\left(\ln\frac{t}{a}\right)^k  \left[\left(t\frac{d}{dt}\right)^kx(t)\right]_{t=a}.$$

\begin{theorem}\label{thm:IntDer} Given a function $x\in C^1[a,b]$, we have
$$\LD\LI x(t)=x(t) \quad \mbox{and} \quad \RD\RI x(t)=x(t).$$
\end{theorem}

\begin{proof} We prove the formula for the left-sided fractional operators only. By definition,
\begin{equation}\label{aux1}\LD\LI x(t)={\mathcal{I}_{a+}^{n-\a,\r}}y_{(n)}(t), \quad \mbox{with} \quad y_{(n)}(t)=\left(t^{1-\r}\frac{d}{dt}\right)^n\LI x(t).
\end{equation}
Computing directly, and since $\a\in(n-1,n)$, we get
$$ \begin{array}{ll}
\DS y_{(1)}(t)&=\DS t^{1-\r}\frac{d}{dt}\frac{\r^{1-\a}}{\Gamma(\a)}\int_a^t \t^{\r-1}(t^\r-\t^\r)^{\a-1}x(\t) d\t\\
&\\
&=\DS\frac{\r^{2-\a}}{\Gamma(\a-1)}\int_a^t \t^{\r-1}(t^\r-\t^\r)^{\a-2}x(\t) d\t.\\
\end{array}$$
Repeating the process, we arrive at the expression
$$ \begin{array}{ll}
\DS y_{(n-1)}(t)&=\DS\frac{\r^{n-\a-1}}{\Gamma(\a-n+1)}\int_a^t \r \t^{\r-1}(t^\r-\t^\r)^{\a-n}x(\t) d\t.\\
&\\
&=\DS\frac{\r^{n-\a-1}}{\Gamma(\a-n+2)}\left[(t^\r-a^\r)^{\a-n+1}x(a)+\int_a^t (t^\r-\t^\r)^{\a-n+1}\frac{d}{d\t}x(\t) d\t\right].\\
\end{array}$$
using integration by parts. Then, we finally arrive to
$$y_{(n)}(t)=t^{1-\r}\frac{d}{dt}y_{(n-1)}(t)=\frac{\r^{n-\a}}{\Gamma(\a-n+1)}\left[(t^\r-a^\r)^{\a-n}x(a)+\int_a^t (t^\r-\t^\r)^{\a-n}\frac{d}{d\t}x(\t) d\t\right].$$
Then, replacing this last expression into equation \eqref{aux1}, we obtain
$$\LD\LI x(t)=\frac{\r}{\Gamma(n-\a)\Gamma(\a-n+1)}\left[x(a)\int_a^t\t^{\r-1}(t^\r-\t^\r)^{n-\a-1}(\t^\r-a^\r)^{\a-n}d\t\right.$$
$$\left.+\int_a^t \int_a^\t \t^{\r-1}(t^\r-\t^\r)^{n-\a-1}(\t^\r-s^\r)^{\a-n}\frac{d}{ds}x(s) ds\,d\t\right].$$
With the change of variables $u=(\t^\r-a^\r)/(t^\t-a^\r)$, we get
$$\int_a^t\t^{\r-1}(t^\r-\t^\r)^{n-\a-1}(\t^\r-a^\r)^{\a-n}d\t
=\int_a^t\t^{\r-1}(t^\r-a^\r)^{n-\a-1}\left(1-\frac{\t^\r-a^\r}{t^\r-a^\r}\right)^{n-\a-1}(\t^\r-a^\r)^{\a-n}d\t$$
$$=\frac1\r\int_0^1(1-u)^{n-\a-1}u^{\a-n}du=\frac1\r B(n-\a,\a-n+1)=\frac{\Gamma(n-\a)\Gamma(\a-n+1)}{\r}.$$
In a similar way, and using the Dirichlet's formula, we get
$$\int_a^t \int_a^\t \t^{\r-1}(t^\r-\t^\r)^{n-\a-1}(\t^\r-s^\r)^{\a-n}\frac{d}{ds}x(s) ds\,d\t
=\int_a^t \int_\t^t s^{\r-1}(t^\r-s^\r)^{n-\a-1}(s^\r-\t^\r)^{\a-n}\frac{d}{d\t}x(\t) ds\,d\t$$
$$=\int_a^t\frac{d}{d\t}x(\t)\frac{\Gamma(n-\a)\Gamma(\a-n+1)}{\r}d\t=\frac{\Gamma(n-\a)\Gamma(\a-n+1)}{\r}(x(t)-x(a)).$$
Then,
$$\LD\LI x(t)=\frac{\r}{\Gamma(n-\a)\Gamma(\a-n+1)}\left[x(a)\frac{\Gamma(n-\a)\Gamma(\a-n+1)}{\r}\right.$$
$$\left.+\frac{\Gamma(n-\a)\Gamma(\a-n+1)}{\r}(x(t)-x(a))\right]=x(t).$$
\end{proof}

Again, for $\r=1$ and $\r\to0^+$, we recover the classical formulas as in Lemma 2.21 of \cite{Kilbas} and in Lemma 2.4 of \cite{Baleanu2},  respectively:
$${^CD_{a+}^\a}{I_{a+}^\a} x(t)={^{CH}D_{a+}^\a}{^H I_{a+}^\a} x(t)=x(t).$$
We now establish a relation between the Katugampola and the Caputo--Katugampola fractional derivatives.

\begin{theorem}\label{thm:DerR-C} Let $x\in C^n[a,b]$ be a function. Then
$$ \LD x(t)=\LDRL \left[x(t)-\sum_{k=0}^{n-1}\frac{\r^{-k}}{k!}(t^\r-a^\r)^{k}x_{(k)}(a)\right]$$
and
$$ \RD x(t)=\RDRL \left[x(t)-\sum_{k=0}^{n-1}\frac{\r^{-k}(-1)^k}{k!}(b^\r-t^\r)^{k}x_{(k)}(b)\right].$$
\end{theorem}

\begin{proof} Starting with the definition of the Katugampola fractional derivative, and integrating by parts, one deduces
$$\left(t^{1-\r}\frac{d}{dt}\right)^n\int_a^t \t^{\r-1}(t^\r-\t^\r)^{n-\a-1}
\left[x(\t)-\sum_{k=0}^{n-1}\frac{\r^{-k}}{k!}(\t^\r-a^\r)^{k}x_{(k)}(a)\right] d\tau$$
$$=\left(t^{1-\r}\frac{d}{dt}\right)^n  \int_a^t \frac{\t^{\r-1}(t^\r-\t^\r)^{n-\a}}{\r(n-\a)}
\left[x_{(1)}(\t)-\sum_{k=1}^{n-1}\frac{\r^{1-k}}{(k-1)!}(\t^\r-a^\r)^{k-1}x_{(k)}(a)\right] d\tau$$
$$=\left(t^{1-\r}\frac{d}{dt}\right)^{n-1}  \int_a^t\t^{\r-1}(t^\r-\t^\r)^{n-\a-1}
\left[x_{(1)}(\t)-\sum_{k=1}^{n-1}\frac{\r^{1-k}}{(k-1)!}(\t^\r-a^\r)^{k-1}x_{(k)}(a)\right] d\tau.$$
Repeating the process $n-2$ more times, we arrive at the equivalent expression
$$t^{1-\r}\frac{d}{dt}\int_a^t \frac{\t^{\r-1}(t^\r-\t^\r)^{n-\a}}{\r(n-\a)}x_{(n)}(\t) d\tau
=\int_a^t \t^{\r-1}(t^\r-\t^\r)^{n-\a-1}x_{(n)}(\t) d\tau,$$
proving the first formula. The second one is obtained in a similar way.
\end{proof}

Formula obtained in Theorem \ref{thm:DerR-C} allows us to deduce a direct relation between the two types of fractional derivative operators. In fact, similar as done before, we have
$$\LDRL (t^\r-a^\r)^{k}=\frac{\r^{-n+\a}k!}{\Gamma(n-\alpha+k+1)}\left(t^{1-\r}\frac{d}{dt}\right)^n(t^\r-a^\r)^{n-\a+k},$$
and since
$$\left(t^{1-\r}\frac{d}{dt}\right)^n(t^\r-a^\r)^{n-\a+k}=\r^n\frac{\Gamma(n-\a+k+1)}{\Gamma(k+1-\a)}(t^\r-a^\r)^{k-\a},$$
we get the following relation
$$ \LD x(t)=\LDRL x(t)-\sum_{k=0}^{n-1}\frac{\r^{\a-k}}{\Gamma(k+1-\a)}(t^\r-a^\r)^{k-\a}x_{(k)}(a).$$
Analogously, we obtain
$$ \RD x(t)=\RDRL x(t)-\sum_{k=0}^{n-1}\frac{\r^{\a-k}(-1)^k}{\Gamma(k+1-\a)}(b^\r-t^\r)^{k-\a}x_{(k)}(b).$$

The following result establishes an integration by parts formula, and generalizes the formula proven in \cite{Almeida2} for arbitrary real $\a>0$.

\begin{theorem} Let $x\in C[a,b]$ and $y\in C^n[a,b]$ be two functions. Then,
$$\int_a^b x(t) \, \LD y(t) \, dt=\int_a^b \RDRL(t^{1-\r}x(t))\, t^{\r-1}y(t) \, dt$$
$$+\left[\sum_{k=0}^{n-1}\left(-t^{1-\r}\frac{d}{dt}\right)^k  {\mathcal{I}_{b-}^{n-\a,\r}}(t^{1-\r}x(t))\, y_{(n-k-1)}(t)    \right]_{t=a}^{t=b},$$
and
$$\int_a^b x(t) \, \RD y(t) \, dt=\int_a^b \LDRL(t^{1-\r}x(t))\, t^{\r-1}y(t) \, dt
$$
$$+\left[\sum_{k=0}^{n-1}(-1)^{n-k}\left(t^{1-\r}\frac{d}{dt}\right)^k  {\mathcal{I}_{a+}^{n-\a,\r}}(t^{1-\r}x(t))\, y_{(n-k-1)}(t)    \right]_{t=a}^{t=b}.$$
\end{theorem}

\begin{proof}
We prove only the first formula; the second one is similar. Applying the Dirichlet's formula and integrating by parts, we get

$$\begin{array}{ll}
\DS \int_a^b x(t) \, \LD y(t) \, dt & =\DS
\frac{\r^{1-n+\a}}{\Gamma(n-\alpha)}\int_a^b\int_a^t x(t)(t^\r-\t^\r)^{n-\a-1}\frac{d}{d\t}y_{(n-1)}(\t)\, d\tau\,dt\\
&\DS = \frac{\r^{1-n+\a}}{\Gamma(n-\alpha)}\int_a^b\int_t^b x(\t)(\t^\r-t^\r)^{n-\a-1}\, d\t\, \frac{d}{dt}y_{(n-1)}(t)\,dt\\
&\DS = \frac{\r^{1-n+\a}}{\Gamma(n-\alpha)}\left[\int_t^b x(\t)(\t^\r-t^\r)^{n-\a-1}\, d\t\, y_{(n-1)}(t)\right]_{t=a}^{t=b}\\
&\DS - \frac{\r^{1-n+\a}}{\Gamma(n-\alpha)}\int_a^b\frac{d}{dt}\left(\int_t^b x(\t)(\t^\r-t^\r)^{n-\a-1}\, d\t\right)\, t^{1-\r}\frac{d}{dt}y_{(n-2)}(t)\,dt\\
&\DS = \left[ {\mathcal{I}_{b-}^{n-\a,\r}}(t^{1-\r}x(t)) \, y_{(n-1)}(t)\right]_{t=a}^{t=b}\\
&\DS +\int_a^b -t^{1-\r}\frac{d}{dt}  {\mathcal{I}_{b-}^{n-\a,\r}}(t^{1-\r}x(t)) \, \frac{d}{dt}y_{(n-2)}(t)\,dt.\\
\end{array}$$
Integrating once more by parts, we obtain
$$\begin{array}{ll}
\DS \int_a^b x(t) \, \LD y(t) \, dt
&\DS = \left[ \sum_{k=0}^1 \left(-t^{1-\r}\frac{d}{dt}\right)^k{\mathcal{I}_{b-}^{n-\a,\r}}(t^{1-\r}x(t)) \, y_{(n-k-1)}(t)\right]_{t=a}^{t=b}\\
&\DS + \int_a^b\left(-t^{1-\r}\frac{d}{dt}\right)^2  {\mathcal{I}_{b-}^{n-\a,\r}}(t^{1-\r}x(t)) \, \frac{d}{dt}y_{(n-3)}(t)\,dt.
\end{array}$$
If we integrate by parts $n-3$ more times, we get
$$\begin{array}{ll}
\DS \int_a^b x(t) \, \LD y(t) \, dt
&\DS = \left[ \sum_{k=0}^{n-2} \left(-t^{1-\r}\frac{d}{dt}\right)^k{\mathcal{I}_{b-}^{n-\a,\r}}(t^{1-\r}x(t)) \, y_{(n-k-1)}(t)\right]_{t=a}^{t=b}\\
&\DS + \int_a^b\left(-t^{1-\r}\frac{d}{dt}\right)^{n-1}  {\mathcal{I}_{b-}^{n-\a,\r}}(t^{1-\r}x(t)) \, \frac{d}{dt}y(t)\,dt.
\end{array}$$
The formula follows integrating by parts once more the last integral.
\end{proof}

\section{The Gronwall inequality}\label{sec:Gronwall}

The Gronwall inequality plays a central role in the theory of differential equations, since it allows to estimate the difference between two solutions of two differential equations $\dot{x}(t)=f(t,x(t))$ and $\dot{x}(t)=g(t,x(t))$, in terms of the difference of the two initial conditions for each of the two differential equations, and the difference between the two dynamic equations  $f$ and $g$ (see e.g. \cite{Dragomir}).
Recently, the Gronwall inequality has been generalized for the study of fractional differential equations, with dependence on the Riemann--Liouville fractional derivative \cite{Ye} and for the Hadamard fractional derivative \cite{Gong}. Here we present a more general form, valid for the Katugampola fractional derivative.

\begin{theorem}  Let $u,v$ be two integrable functions and $g$ a continuous function, with domain $[a,b]$. Assume that
\begin{enumerate}
\item $u$ and $v$ are nonnegative;
\item $g$ is nonnegative and nondecreasing.
\end{enumerate}
If
$$u(t)\leq v(t)+g(t)\r^{1-\a}\int_a^t\t^{\r-1}(t^\r-\t^\r)^{\a-1}u(\t)\,d\t, \quad \forall t\in[a,b],$$
then
$$u(t)\leq v(t)+\int_a^t\sum_{k=1}^\infty\frac{\r^{1-k\a}(g(t)\Gamma(\a))^k}{\Gamma(k\a)}\t^{\r-1}(t^\r-\t^\r)^{k\a-1}v(\t)\,d\t, \quad \forall t\in[a,b].$$
In addition, if $v$ is nondecreasing, then
$$u(t)\leq v(t)E_\a\left[g(t)\Gamma(\a)\left(\frac{t^\r-a^\r}{\r}\right)^\a\right], \quad \forall t\in[a,b].$$
\end{theorem}

\begin{proof} Define the functional
$$\Psi x = g(t)\r^{1-\a}\int_a^t \t^{\r-1}(t^\r-\t^\r)^{\a-1}x(\t)\,d\t.$$
Then $u(t)\leq v(t)+\Psi u(t)$. Iterating consecutively,  we obtain for $n\in\mathbb N$,
$$u(t)\leq \sum_{k=0}^{n-1}\Psi^k v(t)+\Psi^n u(t).$$
Let us prove, by mathematical induction, that if $x$ is a nonnegative function, then
$$\Psi^k x(t) \leq \r^{1-k\a}\int_a^t \frac{(g(t)\Gamma(\a))^k}{\Gamma(k\a)}\t^{\r-1}(t^\r-\t^\r)^{k\a-1}x(\t)\,d\t.$$
For $k=1$ is obvious. Suppose that the formula if valid for $k\in\mathbb N$. Then,
$$\Psi^{k+1} x(t)=\Psi\Psi^k x(t) \leq g(t)\r^{1-\a}\int_a^t \t^{\r-1}(t^\r-\t^\r)^{\a-1}
 \r^{1-k\a}\int_a^\t \frac{(g(\t)\Gamma(\a))^k}{\Gamma(k\a)}s^{\r-1}(\t^\r-s^\r)^{k\a-1}x(s)\,ds\,d\t.$$
 Since $g$ is nondecreasing, $g(\t)\leq g(t)$, for all $\t\leq t$, and so
 $$\Psi^{k+1} x(t)\leq (g(t))^{k+1}\r^{2-(k+1)\a} \frac{(\Gamma(\a))^k}{\Gamma(k\a)} \int_a^t \int_a^\t  \t^{\r-1}(t^\r-\t^\r)^{\a-1}
s^{\r-1}(\t^\r-s^\r)^{k\a-1}x(s)\,ds\,d\t.$$
Using the  Dirichlet's formula, we get
$$\Psi^{k+1} x(t)\leq (g(t))^{k+1}\r^{2-(k+1)\a} \frac{(\Gamma(\a))^k}{\Gamma(k\a)} \int_a^t \t^{\r-1}x(\t) \int_\t^t  s^{\r-1}(t^\r-s^\r)^{\a-1}
(s^\r-\t^\r)^{k\a-1}\,ds\,d\t.$$
Evaluating the inner integral in a similar way as was done in the proof of Theorem \ref{thm:IntDer}, we obtain
$$\int_\t^t  s^{\r-1}(t^\r-s^\r)^{\a-1}(s^\r-\t^\r)^{k\a-1}\,ds=\frac{\Gamma(\a)\Gamma(k\a)}{\r\Gamma(k\a+\a)}(t^\r-\t^\r)^{(k+1)\a-1}.$$
Then,
$$\Psi^{k+1} x(t) \leq \r^{1-(k+1)\a}\int_a^t \frac{(g(t)\Gamma(\a))^{k+1}}{\Gamma((k+1)\a)}\t^{\r-1}(t^\r-\t^\r)^{(k+1)\a-1}x(\t)\,d\t,$$
proving the desired. Let us prove now that $\Psi^n u(t)\to0$ as $n\to\infty$. First, using the continuity of $g$ in the interval $[a,b]$, we ensure the existence of a constant $M>0$ such that $g(t)\leq M$, for all $t\in[a,b]$. Then
$$0\leq \Psi^n u(t) \leq \r^{1-n\a}\int_a^t \frac{(M\Gamma(\a))^n}{\Gamma(n\a)}\t^{\r-1}(t^\r-\t^\r)^{n\a-1}u(\t)\,d\t.$$
Consider the series
$$\sum_{n=1}^\infty \frac{(M\Gamma(\a))^n}{\Gamma(n\a)}.$$
If we apply the ratio test to the series, and the asymptotic approximation
$$\lim_{n\to\infty}\frac{\Gamma(n\a)(n\a)^\a}{\Gamma(n\a+\a)}=1,$$
we get
$$ \lim_{n\to\infty}\frac{\Gamma(n\a)}{\Gamma(n\a+\a)}=0.$$
Thus, the series converges and therefore $\Psi^n u(t)\to0$ as $n\to\infty$.
In conclusion, we have
$$u(t)\leq \sum_{k=0}^\infty\Psi^k v(t)\leq v(t)+\int_a^t\sum_{k=1}^\infty\frac{\r^{1-k\a}(g(t)\Gamma(\a))^k}{\Gamma(k\a)}\t^{\r-1}(t^\r-\t^\r)^{k\a-1}v(\t)\,d\t.$$
For the second case, suppose now that $v$ is nondecreasing. So, for all $\t\in[a,t]$, we have $v(\t)\leq v(t)$ and so
$$\begin{array}{ll}
u(t)& \leq \DS v(t)\left[1+\sum_{k=1}^\infty\frac{\r^{1-k\a}(g(t)\Gamma(\a))^k}{\Gamma(k\a)}\int_a^t\t^{\r-1}(t^\r-\t^\r)^{k\a-1}\,d\t\right]\\
& = \DS v(t)\left[1+\sum_{k=1}^\infty\frac{\r^{-k\a}(g(t)\Gamma(\a)(t^\r-a^\r)^\a)^k}{\Gamma(k\a+1)}\right]\\
&=\DS v(t)E_\a\left[g(t)\Gamma(\a)\left(\frac{t^\r-a^\r}{\r}\right)^\a\right].
\end{array}$$
\end{proof}

For the right fractional operator, the following result is proven in a similar way.

\begin{theorem}  Let $u,v$ be two integrable functions and $g$ a continuous function, with domain $[a,b]$. Assume that
\begin{enumerate}
\item $u$ and $v$ are nonnegative;
\item $g$ is nonnegative and nonincreasing.
\end{enumerate}
If
$$u(t)\leq v(t)+g(t)\r^{1-\a}\int_t^b\t^{\r-1}(\t^\r-t^\r)^{\a-1}u(\t)\,d\t, \quad \forall t\in[a,b],$$
then
$$u(t)\leq v(t)+\int_t^b\sum_{k=1}^\infty\frac{\r^{1-k\a}(g(t)\Gamma(\a))^k}{\Gamma(k\a)}\t^{\r-1}(\t^\r-t^\r)^{k\a-1}v(\t)\,d\t, \quad \forall t\in[a,b].$$
In addition, if $v$ is nonincreasing, then
$$u(t)\leq v(t)E_\a\left[g(t)\Gamma(\a)\left(\frac{b^\r-t^\r}{\r}\right)^\a\right], \quad \forall t\in[a,b].$$
\end{theorem}

Using the Gronwall inequality, we can relate  solutions of two fractional differential equations. Consider the following fractional differential equation
\begin{equation}\label{FDE}\left\{
\begin{array}{l}
\LD x(t)=f(t,x(t))\\
x^{(i)}(a)=x_a^i, \quad i=0,\ldots,n-1,
\end{array}  \right.\end{equation}
where $f:[a,b]\times\mathbb R\to\mathbb R$ is a continuous function, $\a\in(n-1,n)$ and $x_a^i$ are fixed reals, for $i=0,\ldots,n-1$. Applying the fractional integral operator $\LI$ to both sides of the fractional differential  equation in system \eqref{FDE} and using Theorem \ref{thm:DerInt}, we get
\begin{equation}\label{FDE:Volterra}\begin{array}{ll}
x(t)&=\DS\sum_{k=0}^{n-1}\frac{\r^{-k}}{k!}(t^\r-a^\r)^kx_{(k)}(a)+\LI f(t,x(t))\\
&=\DS\sum_{k=0}^{n-1}\frac{\r^{-k}}{k!}(t^\r-a^\r)^kx_{(k)}(a)+\frac{\r^{1-\a}}{\Gamma(\a)}\int_a^t \t^{\r-1}(t^\r-\t^\r)^{\a-1}f(\t,x(\t)) d\t.
\end{array}\end{equation}
Conversely, if $x$ satisfies Equation \eqref{FDE:Volterra}, then $x$ satisfies system \eqref{FDE}. This is proven applying the fractional derivative operator $\LD$ to both sides of Equation \eqref{FDE:Volterra}, using Theorem \ref{thm:IntDer} and the formula
$$\LD (t^\r-a^\r)^k=0, \quad \forall k\in\{0,1,\ldots,n-1\}.$$

\begin{theorem} Let $f,g:[a,b]\times\mathbb R\to\mathbb R$ be two continuous functions, and $x,y$ solutions of the two following systems
$$\left\{
\begin{array}{l}
\LD x(t)=f(t,x(t))\\
x^{(i)}(a)=x_a^i, \quad i=0,\ldots,n-1,
\end{array}  \right.$$
and
$$\left\{
\begin{array}{l}
\LD y(t)=g(t,y(t))\\
y^{(i)}(a)=y_a^i, \quad i=0,\ldots,n-1.
\end{array}  \right.$$
Suppose that there exist
\begin{enumerate}
\item a positive constant $C$  such that
$$|g(t,y_1)-g(t,y_2)|\leq C |y_1-y_2|, \quad \forall t\in[a,b]\,\forall y_1,y_2\in\mathbb R;$$
\item a continuous function $\psi:[a,b]\to\mathbb R_0^+$ such that
$$|f(t,x(t))-g(t,x(t))|\leq \psi(t), \quad \forall t\in[a,b].$$
\end{enumerate}
Define the function $v:[a,b]\to\mathbb R$ by
$$v(t)=\sum_{k=0}^{n-1}\frac{\r^{-k}}{k!}(t^\r-a^\r)^k\left|x_{(k)}(a)-y_{(k)}(a)\right|+
\frac{\r^{1-\a}}{\Gamma(\a)}\int_a^t \t^{\r-1}(t^\r-\t^\r)^{\a-1}\psi(\t) d\t.$$
Then, for all $t\in[a,b]$,
$$|x(t)-y(t)|\leq v(t)+\int_a^t\sum_{k=1}^\infty\frac{\r^{1-k\a}C^k}{\Gamma(k\a)}\t^{\r-1}(t^\r-\t^\r)^{k\a-1}v(\t)\,d\t.$$
\end{theorem}

\begin{proof} Define $u(t)=|x(t)-y(t)|$. Then
$$\begin{array}{ll}
u(t)&\DS\leq \sum_{k=0}^{n-1}\frac{\r^{-k}}{k!}(t^\r-a^\r)^k\left|x_{(k)}(a)-y_{(k)}(a)\right|\\
&  \quad \DS+\frac{\r^{1-\a}}{\Gamma(\a)}\int_a^t \t^{\r-1}(t^\r-\t^\r)^{\a-1}|f(\t,x(\t))-g(\t,y(\t))| d\t\\
&\DS\leq \sum_{k=0}^{n-1}\frac{\r^{-k}}{k!}(t^\r-a^\r)^k\left|x_{(k)}(a)-y_{(k)}(a)\right|\\
&  \quad \DS+\frac{\r^{1-\a}}{\Gamma(\a)}\int_a^t \t^{\r-1}(t^\r-\t^\r)^{\a-1}\left(|f(\t,x(\t))-g(\t,x(\t))|+|g(\t,x(\t))-g(\t,y(\t))| \right)d\t.
\end{array}$$
Using the relations
$$|f(\t,x(\t))-g(\t,x(\t))|\leq \psi(\t)\quad \mbox{and} \quad |g(\t,x(\t))-g(\t,y(\t))|\leq C |x(\t)-y(\t)|,$$
and the Gronwall inequality, we prove the result.
\end{proof}

In particular, when $g=f$, we obtain a simpler formula:
\begin{equation}\label{compsolu}\begin{array}{ll}
|x(t)-y(t)|&\DS\leq \sum_{k=0}^{n-1}\frac{\r^{-k}}{k!}(t^\r-a^\r)^k\left|x_{(k)}(a)-y_{(k)}(a)\right|\\
&  \quad \DS+\int_a^t\sum_{k=1}^\infty\frac{\r^{1-k\a}C^k}{\Gamma(k\a)}\t^{\r-1}(t^\r-\t^\r)^{k\a-1}
\sum_{k=0}^{n-1}\frac{\r^{-k}}{k!}(\t^\r-a^\r)^k\left|x_{(k)}(a)-y_{(k)}(a)\right|\,d\t.
\end{array}\end{equation}

Also, from Equation \eqref{compsolu}, we see that the solution of system \eqref{FDE} is unique.


\section*{Acknowledgments}

The author is very grateful to an anonymous referee, for valuable remarks and comments that improved this paper.
Work supported by Portuguese funds through the CIDMA - Center for Research and Development in Mathematics and Applications, and the Portuguese Foundation for Science and Technology (FCT-Funda\c{c}\~ao para a Ci\^encia e a Tecnologia), within project UID/MAT/04106/2013.


\end{document}